\RequirePackage{fix-cm}
\documentclass[smallextended]{svjour3}       
\smartqed  
\usepackage{graphicx}
\usepackage{float}
\usepackage{amsfonts}
\usepackage{epstopdf}
\usepackage{amsmath}
\usepackage{mathrsfs}
\usepackage[center]{caption}
\usepackage[linkcolor=blue,colorlinks=true]{hyperref}
\begin{document}

\title{Higher order Hermite-Fej\'{e}r Interpolation on the unit circle.
}
\subtitle{Higher order Hermite-Fej\'{e}r Interpolation.}


\author{Swarnima Bahadur         \and
 Varun 
}


\institute{S. Bahadur \at
               Department of Mathematics \& Astronomy, University of Lucknow, Lucknow, India\\              
              \email{swarnimabahadur@ymail.com}  \and
           Varun \at
              Department of Mathematics \& Astronomy, University of Lucknow, Lucknow, India\\
              \email{varun.kanaujia.1992@gmail.com}
}

\date{Received: date / Accepted: date}

\maketitle

\begin{abstract}
The aim of this paper is to study the approximation of functions using a higher-order Hermite-Fej\'{e}r interpolation process on the unit circle. The system of nodes is composed of vertically projected zeros of Jacobi polynomials onto the unit circle with boundary points at $ \pm1 $. Values of the polynomial and its first four derivatives are fixed by the interpolation conditions at the nodes. Convergence of the process is obtained for analytic functions on a suitable domain, and the rate of convergence is estimated. 
\end{abstract}

\section{Introduction\label{s1}}

Approximation of continous functions can be done using different methods by constructing algebraic or trigonometric polynomials. Hermite interpolation attracted the attention of many researchers in the last century. \\
\textbf{Hermite interpolation} \cite{10}: It is the process of finding a polynomial which coincides with the continous function at certain pre-assigned points, called the nodes of interpolation, and its successsive derivatives coinciding with arbritarily chosen numbers.\\

An important step was taken when Fej\'{e}r \cite{7} in 1916 proved a theorem where the values of the derivatives in the Hermite scheme were equal to zero.\\\\ 
\textbf{Fej\'{e}r's theorem :}   If $ f\in C[-1,1]$, then $ H_{n}(f,x) $ converges to $ f(x) $ uniformly on [-1,1] as $ n $ tends to infinty. Interpolation polynomials $H_{n}(f,x)$ is defined by 
\begin{equation*}
	H_{n}(f,x)=\sum_{k=1}^{n}f(x_{kn})(1-x_{kn}x)\Bigg(\frac{T_n(x)}{n(x-x_{kn})}\Bigg)^2,
\end{equation*}
where $x_{kn}$ are the zeros of the Chebyshev polynomial of the first kind. \\$ H_n(f,x) $ satisfies the below given conditions where $ k=1,2,...,n $.
\begin{eqnarray}\notag
	H_n(f,x_{kn})=f(x_{kn})\quad and \quad H_n'(f,x_{kn})=0.
\end{eqnarray}
Mills \cite{9} in his paper highlights Hermite and Hemite Fej\'{e}r interpolation as important techniques in the approximation theory. Fej\'{e}r's theorem has been extended to more general nodal systems. For example, in 2001, Daruis and Gonz\'{a}lez-Vera \cite{6} extended Fej\'{e}r's result to the unit circle by considering the nodal system constituted by the complex $ n^{th} $ roots of unity. They proved that the sequence of Hermite-Fej\'{e}r interpolation polynomials uniformly converge for continous functions on the unit circle.\\

Berriochoa, Cachafeiro and Garc\'{i}a-Amor \cite{2} extended the Fej\'{e}r's second theorem to the unit circle. Then Berriochoa, Cachafeiro, D\'{i}az, and Mart\'{i}nez Brey \cite{3} obtained the supremum norm of the error of interpolation for analytic functions and computed the order of convergence of Hermite-Fej\'{e}r interpolation on the unit circle considering the same set of nodes as of \cite{6}.\\

Apart from the uniform nodal system (where nodes are equally spaced on the unit circle), Hermite-Fej\'{e}r interpolation on the unit circle have been also studied on some non-uniformly distributed nodes on the unit circle (see \cite{1} and \cite{5}).\\
\textbf{Higher order Hermite-Fejér interpolation: } It is a process of finding a polynomial which coincides with a continous function at the nodes of the interpolation and the derivatives upto $  r^{th}  $ order $ (r>1) $ are null at the nodal points.\\

A considerable number of papers on higher order Hermite-Fejér interpolation processes on real nodes have been published (see \cite{11} and \cite{14}). This motivated us to consider a higher order Hermite-Fej\'{e}r interpolation problem on non-uniform set of complex nodes on the unit circle. Let us denote nodal system containing the zeros of the Jacobi polynomial $ P_n^{(\alpha,\beta)}(x) $ by Gauss Jacobi point system. Let us also define two sets $ \mathbb{T}=\{z:|z|=1\} $ and $ \mathbb{D}=\{z:|z|<1\} $.\\

In the present paper, we consider a Hermite-Fej\'{e}r interpolation problem on the nodal system constituted of $  \pm1 $ and the projections of the Gauss Jacobi point system vertically onto the unit circle by the transformation {$x=\dfrac{1+z^2}{2z}$}. The aim of this paper is to extend the Hermite-Fej\'{e}r interpolation on the unit circle problem on all the above said projected nodes upto the fourth derivative and prove the following convergence theorem:
\begin{theorem}\label{thm1.1}
	Let $f(z)$ is a function continous on $\mathbb{T}\cup\mathbb{D}$ and analytic on $ \mathbb{D} $. The sequence of interpolatory polynomial $ \{Q_n(z)\} $ satisfies the below relation
	\begin{equation}\label{1.2}
		\mid Q_n(z)-f(z) \mid \,=\,  \textbf{O}\big(\omega(f,n^{-1})\log n\big) 
	\end{equation}
	where $ \omega(f,n^{-1}) $ represents the modulus of continuity of the function $ f(z) $ and $ \textbf{O} $ notation refers to as $ n \rightarrow \infty $	.
\end{theorem}
The paper has been organised in following manner.
Preliminaries are given in section \ref{2}. Section \ref{3} covers the interpolation problem and  explicit representation of the interpolatory polynomial. Section \ref{4} is devoted to finding estimates and the proof of theorem \ref{thm1.1} has been assigned section \ref{5}. 

\section{Preliminaries}\label{2}
The differential equation satisfied by  $ P_n^{(\alpha,\beta)}(x) $ is
\begin{equation*}
	(1-x^2)P_n^{(\alpha,\beta)^{''}}(x)+[\beta-\alpha-(\alpha+\beta+2)x]P_n^{(\alpha,\beta)^{'}}(x)+n(n+\alpha+\beta+1)P_n^{(\alpha,\beta)}(x)=0.\\
\end{equation*}
Using the Szeg\H{o} transformation $x=\dfrac{1+z^2}{2z}$,
\begin{equation}
	\begin{split}\label{2.1}
		(z^2-1)^4P_n^{(\alpha,\beta)^{''}}(x)+4z(z^2-1)&\big[\{(\alpha+\beta+2)z^2+1\}(z^2-1)-2z^3(\beta-\alpha)\big]P_n^{(\alpha,\beta)^{'}}(x)\\
		-16z^6n(n+\alpha+\beta+1)P_n^{(\alpha,\beta)}(x)&=0.\\
	\end{split}
\end{equation}

Let ${Z}_n $ be set of nodes
\begin{equation}\notag
	\begin{split}
		{Z}_n =\{z_0=1, z_{2n+1}=-1, z_k=x_k+iy_k=\cos\theta_k+i\, \sin\theta_k\,;\,z_{n+k}=\overline{z_k};\\\,\, k=1,2,3,...,n\,;\,x_k,y_k\in R\},\\
	\end{split}
\end{equation}
which are obtained by projecting vertically the Gauss Jacobi point system on the unit circle together with $ \pm1 $.

The polynomial defined on $ Z_n $  are given by (\ref{2.2}),
\begin{figure}
	\caption{An arbitrary point $ z $ and the nodal system $ Z_n $}
	\centering
	\includegraphics[width=0.5\textwidth]{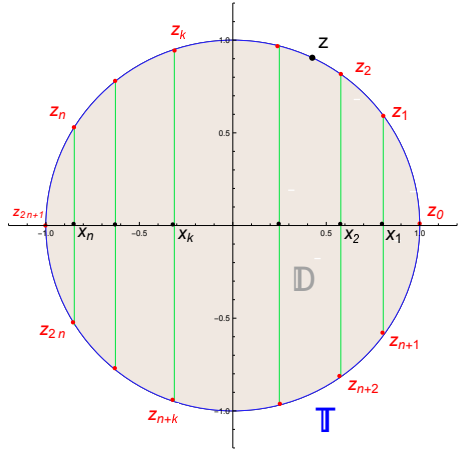}
\end{figure}
\begin{equation}\label{2.2}
	R(z)=(z^2-1)W(z),
\end{equation}
where
\begin{equation}\label{2.3}
	W(z)=\displaystyle\prod_{k=1}^{2n}(z-z_k)=K_nP_n^{(\alpha,\beta)}\Bigg(\frac{1+z^2}{2z}\Bigg)z^n, 
\end{equation}
\begin{equation*}
	K_n=2^{2n}n!\frac{\Gamma(\alpha+\beta+n+1)}{\Gamma(\alpha+\beta+2n+1)}.
\end{equation*}

The fundamental polynomials of Lagrange interpolation on the zeros of $ R(z) $ are given by 
\begin{equation}\label{2.4}
	L_{k}(z)=\frac{R(z)}{(z-z_k)R^{'}(z_k)},\hspace{1cm}k=0,1,...,2n+1.
\end{equation}\\
We can write $ z=x+iy $, where $ x,y\in R $. If $ z\in\mathbb{T}$, then
\begin{equation}\label{2.5}	
	\mid z^2-1\mid=2\sqrt{1-x^2}
\end{equation}
and
\begin{equation}\label{2.6}
	\mid z-z_k\mid=\sqrt{2}\,\sqrt{1-xx_k-\sqrt{1-x^2}\sqrt{1-x_k^2}}.
\end{equation}
In order to evaluate the estimates of the fundamental polynomials formed in section \ref{3}, we will be using below results.\\

For $-1\leq x\leq1$, we have
\begin{equation}\label{2.7}
	(1-x^2)^{1/2}\mid{P_n^{(\alpha,\beta)}(x)}\mid=O(n^{\alpha-1}),
\end{equation}
\begin{equation}\label{2.8}
	\mid{P_n^{(\alpha,\beta)}(x)}\mid=O(n^{\alpha}),
\end{equation}
\begin{equation}\label{2.9}
	\mid{P_n^{(\alpha,\beta)'}(x)}\mid=O(n^{\alpha+2}),
\end{equation}
\begin{equation}\label{2.10}
	\mid{P_n^{(\alpha,\beta)''}(x)}\mid=O(n^{\alpha+4}).
\end{equation}
Considering set of nodes $ Z_n $, where $ x_k= \cos\theta_k \, ,k=1,2,...,n$ are the zeros of $ {P_n^{(\alpha,\beta)}(x)} $, then
\begin{equation}\label{2.11}
	(1-x_k^2)^{-1}\sim\Bigg(\frac{k}{n}\Bigg)^{-2},
\end{equation}
\begin{equation}\label{2.12}
	\mid{P_n^{(\alpha,\beta)}(x_k)}\mid\sim k^{-\alpha-\frac{1}{2}}n^{\alpha},
\end{equation}
\begin{equation}\label{2.13}
	\mid{P_n^{(\alpha,\beta)'}(x_k)}\mid\sim k^{-\alpha-\frac{3}{2}}n^{\alpha+2}.
\end{equation}
For more details, refer pg.164-166 of \cite{13}.\\\\
Let $f(z)$ be continous on $\mathbb{T}\cup\mathbb{D}$ and analytic on $ \mathbb{D} $. Then, there exists a polynomial $F_n(z) $ of degree less than $ (2n+2)(r+1) $ satisfying Jackson's inequality.\cite{8}
\begin{equation}\label{2.14}
	\mid f(z)-F_n(z) \mid \,\leq \,C \,\omega(f,n^{-1}),
\end{equation}
where $ \omega(f,n^{-1}) $ represents the modulus of continuity of the function $ f(z) $ and $ C $ is independent	of $ n $  and $ z $.

\section{The problem and explicit representation of interpolatory polynomial}\label{3}

Here, we are interested in determining the convergence of interpolatory polynomial $ Q_n(z) $ of degree less than $(2n+2)(r+1) $ on the distinct set of nodes $ \{z_k\}^{2n+1}_{k=0} $ with Hermite conditions at all points satisfying 
\begin{equation}\label{3.1}
	\begin{cases}
		Q_n{(z_k)}=\alpha_k,   \,\,\, \,\,\,\,\,\,\,\,\,\,\,\,\ k=0,1,...,2n+1\\
		
		Q_n^{(r)}{(z_k)}=0, \,\,\,\,\,\,\,\,\,\,\,\,\,\,\,\ k=0,1,...,2n+1,\, r=1,2,3,4\,,\\
	\end{cases} 
\end{equation}

where $ \alpha_k $ and $ \beta_k $ are complex constants.

\begin{theorem}
	We  shall write $ Q_n(z) $ satisfying (\ref{3.1})\\ 
	\begin{equation}\label{3.2}
		Q_n(z) = \sum_{k=0}^{2n+1} f(z_k)A_{0k}(z),
	\end{equation}
	
	where $ A_{0k}(z) $ is a polynomial  of degree less than $(2n+2)(r+1) $  satisfying the conditions given in (\ref{3.3}).
	
	For $ j,k=0,1,...,2n+1 $ ,
	\begin{equation}\label{3.3}
		\begin{cases}
			A_{0k}(z_j) = \delta_{kj},\\
			
			A_{0k}^{(r)}(z_j)=0\,\,\,\,\,\,\,\,\,\,\,\,\,\,\,\,\, \, ;r=1,2,3,4\,\,,\\
		\end{cases} 
	\end{equation}
	where	
	
	\begin{equation}\label{3.4}
		A_{0k}(z)= [L_k(z)]^{5} + \sum_{p=1}^{4}c_{pk}A_{pk}(z),
	\end{equation}
	\begin{equation}\label{3.5}
		A_{pk}(z)=[R(z)]^p(L_k(z))^{5-p},
	\end{equation}
	\begin{equation}\label{3.6}
		c_{1k}=-\,\frac{5\,L_k'(z_k)}{R'(z_k)},
	\end{equation}
	
	\begin{equation}\label{3.7}
		c_{2k}=-\,\frac{5}{2!\,[R'(z_k)]^2}\,\big[L_k''(z_k)+10\,[L_k'(z_k)]^2\big],
	\end{equation}
	
	\begin{equation}\label{3.8}
		c_{3k}=-\,\frac{5}{3!\,[R'(z_k)]^3}\,\big[-18\,L_k''(z_k)\,L_k'(z_k)\,+\,L_k'''(z_k)\,-\,198[L_k'(z_k)]^3\big],
	\end{equation}
	\begin{equation}\label{3.9}
		\begin{split}
			c_{4k}=&-\frac{5}{4!\,[R'(z_k)]^4}\bigg[L_k''''(z_k)-24L_k'''(z_k)\,L_k'(z_k)\\
			&+[L_k''(z_k)]^2-156[L_k'(z_k)]^2L_k''(z_k)+2544[L_k'(z_k)]^4\bigg].\\	
		\end{split}
	\end{equation}
\end{theorem}

\begin{proof}
	Let $  A_{0k}(z)  $ be written as  
	\begin{equation}\label{3.10}
		A_{0k}(z)= [L_k(z)]^{5} + \sum_{p=1}^{4}c_{pk}[R(z)]^p(L_k(z))^{5-p}.
	\end{equation} 
	At $ z=z_j $, where $ j=0,1,...,2n+1 $ ,
	\begin{equation*}
		A_{0k}(z_j)= [L_k(z_j)]^{5} + \sum_{p=1}^{4}c_{pk}[R(z_j)]^p(L_k(z_j))^{5-p}.
	\end{equation*} 
	Using (\ref{2.2}), we have $ R(z_j)=0 $ and from (\ref{2.4}), we have 
	
	\begin{equation}\label{3.11}
		A_{0k}(z_j)=\delta_{kj}.
	\end{equation}
	Clearly, the first set of condition in (\ref{3.3}) is satisfied.\\
	In order to determine the $  c_{pk} $'s, we use the second set of conditions of (\ref{3.3}).\\
	On differentiating $ A_{0k}(z)$ in (\ref{3.10}) one time with respect to $z$, we get
	\begin{equation}\label{3.12}
		A_{0k}'(z)= 5L_k'(z)[L_k(z)]^{4} + c_{1k}[R(z)(L_k(z))^4]'+\Big[\sum_{p=2}^{4}c_{pk}[R(z)]^p(L_k(z))^{5-p}\Big]'.
	\end{equation}
	Clearly, at $ z=z_j \,\,(j\neq k)$, we have $ A_{0k}'(z)=0.$\\
	At $ z=z_k $, $ A_{0k}'(z) $ must be equal to zero. We have
	\begin{eqnarray}\notag
		5\,L_k'(z_k)+c_{1k}\,R'(z_k)=0,
	\end{eqnarray} 
	which provides (\ref{3.6}).
	In a similar manner, differentiating (\ref{3.10}) two, three and four times with respect to $ z $ gives (\ref{3.7}), (\ref{3.8}) and (\ref{3.9}) respectively by using conditions given in (\ref{3.3}).
	
\end{proof}

\section{Estimation of the fundamental polynomials}\label{4}

\normalfont In order to find the estimates, we intend to represent the constants $ c_{pk} $ in general form as given under ($ p $=0,1,2,3,4)
\begin{equation}\label{4.1}
	c_{pk}=\frac{5}{p![R'(z_k)]^p}\sum_{s=0}^{[\frac{p}{2}]}\sum_{r=s}^{p-s}e_{psr}[L_k^{(s)}(z_k)]^r
	L_k^{(p-sr)}(z_k),
\end{equation}

where $ e_{psr} $ are the constants independent of $ n $ and $ z $ and $ \bigg[\dfrac{p}{2}\bigg] $ denotes greatest integer function Also, $ L_k^{(s)}(z_k) $ denotes $ s^{th } $ derivative of  $ L_k(z) $ at $ z=z_k.  $ 

\begin{lemma}\label{lem4.1}
	Let $  L_k(z)  $ be given by (\ref{2.4}), then for $ z\in\mathbb{T}\cup\mathbb{D}$, we have 	
	\begin{equation}\label{4.2}
		\mid L_k(z)\mid=\textit{\textbf{O}} \Bigg(\frac{1}{k^{-\alpha+\frac{3}{2}}}\bigg)
	\end{equation}
	\begin{proof}
		For $ k=1,2,...,2n $
		\begin{equation}\label{4.3}
			L_{k}(z)=\frac{R(z)}{(z-z_k)R^{'}(z_k)}.
		\end{equation}
		Taking modulus on the both sides,
		
		\begin{eqnarray}\notag
			\mid L_{k}(z)\mid&=&\frac{\mid R(z)\mid }{\mid z-z_k\mid \mid R^{'}(z_k)\mid},\\\notag
			&=&\frac{\mid (z^2-1)W(z)\mid }{\mid z-z_k\mid \mid \{2zW(z)+(z^2-1)W'(z))\}_{z=z_k}\mid}.\\\notag
		\end{eqnarray}
		Since $ z_k's $ are the zeros of $ W(z) $, using (\ref{2.3}), we get
		\begin{eqnarray}\notag
			\mid L_{k}(z)\mid&=&\frac{\Big| (z^2-1)K_nP_n^{(\alpha,\beta)}\Big(\frac{1+z^2}{2z}\Big)z^n\Big| }{\mid z-z_k\mid \Big| (z_k^2-1)\Big\{K_nP_n^{(\alpha,\beta)}\Big(\frac{1+z^2}{2z}\Big)z^n\Big\}'_{z=z_k}\Big|}\\\notag
			&=&\frac{\Big| (z^2-1)P_n^{(\alpha,\beta)}\Big(\frac{1+z^2}{2z}\Big)z^n\Big| }{\mid z-z_k\mid \Big| (z_k^2-1)\Big\{nz_k^{n-1}P_n^{(\alpha,\beta)}(x_k)+z_k^nP_n^{(\alpha,\beta)'}(x_k)\Big(\frac{z_k^2-1}{2z_k^2}\Big)\Big\}\Big|}\\\notag
			&=&\frac{2|z^2-1||P_n^{(\alpha,\beta)}(x)||z|^n }{\mid z-z_k\mid  |z_k|^{n-2}|z_k^2-1|^2|P_n^{(\alpha,\beta)'}(x_k)|}.\\\notag
		\end{eqnarray}
		Using (\ref{2.5}) and (\ref{2.6}), we get
		\begin{eqnarray}\notag
			\mid L_{k}(z)\mid&=&\frac{2.2\sqrt{1-x^2}|P_n^{(\alpha,\beta)}(x)||z|^n }{4({1-x_k^2}) \sqrt{2}\,\sqrt{1-xx_k-\sqrt{1-x^2}\sqrt{1-x_k^2}}|P_n^{(\alpha,\beta)'}(x_k)|}\\\notag
			&=&\frac{\sqrt{1-x^2}|P_n^{(\alpha,\beta)}(x)||z|^n\,\sqrt{1-xx_k+\sqrt{1-x^2}\sqrt{1-x_k^2}} }{\sqrt{2}({1-x_k^2}) \,\sqrt{(1-xx_k)^2-({1-x^2})({1-x_k^2})}|P_n^{(\alpha,\beta)'}(x_k)|}\\\notag
			&=&\frac{\sqrt{1-x^2}|P_n^{(\alpha,\beta)}(x)||z|^n\,\sqrt{1-xx_k+\sqrt{1-x^2-x_k^2+x^2x_k^2}} }{\sqrt{2}({1-x_k^2}) \,\mid x-x_k\mid |P_n^{(\alpha,\beta)'}(x_k)|}\\\notag	
			&=&\frac{\sqrt{1-x^2}|P_n^{(\alpha,\beta)}(x)||z|^n\,\sqrt{1-xx_k+\sqrt{(1-xx_k)^2-(x-x_k)^2}} }{\sqrt{2}({1-x_k^2}) \,\mid x-x_k\mid |P_n^{(\alpha,\beta)'}(x_k)|}\\\notag
			&\leq&\frac{\sqrt{1-x^2}|P_n^{(\alpha,\beta)}(x)|\,\sqrt{1-xx_k}}{({1-x_k^2}) \,\mid x-x_k\mid |P_n^{(\alpha,\beta)'}(x_k)|}.\\\notag	
		\end{eqnarray}
		
		For $ \mid x-x_k \mid\geq\frac{1}{2}\mid1-x_k^2 \mid  $, we get
		\begin{eqnarray}\notag
			\mid L_{k}(z)\mid&\leq&C\frac{\sqrt{1-x^2}|P_n^{(\alpha,\beta)}(x)|\,}{({1-x_k^2})^{3/2} \, |P_n^{(\alpha,\beta)'}(x_k)|},\\\notag
		\end{eqnarray}
		where $ C $ is constant independent of $ n $ and $ z $.
		Using (\ref{2.7}), (\ref{2.11}) and (\ref{2.13}), we have
		\begin{eqnarray}\label{4.4}
			\mid L_{k}(z)\mid&=&\textit{\textbf{O}}\Bigg(\frac{1}{k^{-\alpha+\frac{3}{2}}} \Bigg).
		\end{eqnarray}
		Similarly, for $ \mid x-x_k \mid\leq\frac{1}{2}\mid1-x_k^2 \mid  $, we get the same result as (\ref{4.4}).\\
		For $ k=0  $ and $ k=2n+1 $, we have 
		\begin{equation}\label{4.5}
			\mid L_0(z)\mid=\mid L_{2n+1}(z)\mid=\textit{\textbf{O}}(1).
		\end{equation}
		From (\ref{4.4}) and (\ref{4.5}), we have Lemma \ref{lem4.1}. 
	\end{proof}
	
\end{lemma}
\begin{lemma}\label{lem4.2}
	Let $ c_{pk} $ be given by (\ref{4.1}), then
	\begin{equation}\label{4.6}
		\mid c_{pk}\mid =\textit{\textbf{O}}\Bigg(\frac{1}{k_n^p\,n^{p(\alpha-1)}\, k^{\,{p}/{2}-p\alpha }}\Bigg).
	\end{equation}
	\begin{proof}
		From (\ref{2.2}) and (\ref{2.3}), we have 
		\begin{equation*}
			R'(z_k)=\Bigg(\frac{K_n}{2}\Bigg)z_k^{n-2}(z_k^2-1)^2P_n^{(\alpha,\beta)'}(x_k).
		\end{equation*}
		Taking modulus on both the sides, we get
		\begin{eqnarray}\notag
			\mid R'(z_k)\mid=\Bigg(\frac{K_n}{2}\Bigg)\mid z_k\mid^{n-2}\,\mid(z_k^2-1)\mid^2\,\mid P_n^{(\alpha,\beta)'}(x_k)\mid.\\\notag
		\end{eqnarray}
		From (\ref{2.5}), (\ref{2.11}) and (\ref{2.13}), we have
		\begin{equation}\label{4.7}
			\mid R'(z_k)\mid=\textit{\textbf{O}}(K_n\, k^{-\alpha+\frac{1}{2}}n^{\alpha}).
		\end{equation}
		Similarly, from (\ref{2.1}), (\ref{2.4}) and (\ref{2.5}), we have
		\begin{equation}\label{4.8}
			\mid L_k^{(s)}(z_k)\mid=\textit{\textbf{O}}(n^s).
		\end{equation}
		Using (\ref{4.7}) and (\ref{4.8}) in (\ref{4.1}), we have Lemma \ref{lem4.2}.

	\end{proof}
\end{lemma}
\begin{lemma}
	Let $ A_{0k}(z)$ be given by (\ref{3.4}) and $ c_{pk} $ given by (\ref{4.1}), then for $ z\in\mathbb{T}\cup\mathbb{D}$,
	\begin{equation}\label{4.9}
		\sum\limits_{k=0}^{2n+1}\mid A_{0k}(z)\mid\;=\textit{\textbf{O}}(\log n),	
	\end{equation}
	where $ -1<\alpha\leq\dfrac{1}{2}. $
\end{lemma}

\begin{proof}
	From (\ref{2.2}) and (\ref{2.3}), we have
	\begin{eqnarray}\label{4.10}
		R(z)=(z^2-1)K_nP_n^{(\alpha,\beta)}\Bigg(\frac{1+z^2}{2z}\Bigg)z^n.
	\end{eqnarray}
	Taking modulus on both the sides and using (\ref{2.5}) and (\ref{2.7}), we have
	\begin{equation}\label{4.11}
		\mid R(z)\mid=\textit{\textbf{O}}(K_nn^{\alpha-1}).
	\end{equation}
	For $ \mid x_k-x\mid\,\geq\dfrac{1}{2}\mid 1-x_k^2\mid  $ and from (\ref{3.4}) and (\ref{3.5}), we have
	\begin{eqnarray}\notag
		\sum\limits_{k=0}^{2n+1} \mid A_{0k}(z)\mid&=&\sum\limits_{k=0}^{2n+1} \mid L_k(z)\mid ^{5} + \sum\limits_{k=0}^{2n+1} \sum_{p=1}^{4}\mid c_{pk}\mid\,\mid R(z)\mid^p\,\mid L_k(z)\mid^{5-p}.\\ \notag
	\end{eqnarray}
	Using (\ref{4.11}), Lemma \ref{lem4.1} and Lemma \ref{lem4.2}, we get
	\begin{eqnarray}\notag
		\sum\limits_{k=0}^{2n+1} \mid A_{0k}(z)\mid
		&=&\textit{\textbf{O}}\Bigg(\sum\limits_{k=0}^{2n+1} \frac{1}{k^{-5\alpha+\frac{15}{2}}} + \sum\limits_{k=0}^{2n+1} \sum_{p=1}^{4}\,\frac{1}{k^{-5\alpha-p+\frac{15}{2}}} \Bigg),\\\notag
		&=&\textit{\textbf{O}}\Bigg(\sum\limits_{k=0}^{2n+1} \sum_{p=0}^{4}\,\frac{1}{k^{-5\alpha-p+\frac{15}{2}}} \Bigg).\\\notag	
	\end{eqnarray}
	Now, we get
	\begin{equation}\label{4.12}
		\sum\limits_{k=0}^{2n+1} \mid A_{0k}(z)\mid=\textit{\textbf{O}}(\log n),	\qquad\quad\qquad\qquad\Bigg\{-1<\alpha\leq\frac{13}{10}-\frac{p}{5}\Bigg\}.
	\end{equation}
	Similarly, for $ \mid x_k-x\mid\leq\frac{1}{2}\mid 1-x_k^2\mid  $, we get the same result.\\
	Since range of $ \alpha $ with $ p=4 $ lies in the intersection of all the cases. Hence, the lemma follows.
\end{proof}

\section{Proof of theorem 1.1}\label{5}

Let $f(z)$ be a function that is  continous on $\mathbb{T}\cup\mathbb{D}$ and analytic on $ \mathbb{D} $.
Since $Q_n(z)$ is the uniquely determined polynomial of degree less than $ (2n+2)(r+1) $ and the polynomial $F_n(z)$ satisfying equation (\ref{2.14}) can be expressed as 
\begin{equation}\label{5.1}
	F_n(z)=\sum\limits_{k=0}^{2n+1}F_n(z_k)A_k(z).
\end{equation}
Then
\begin{equation}\label{5.2}
	\mid Q_n(z)-f(z)\mid\, \leq\, \mid Q_n(z)-F_n(z)\mid +\mid F_n(z)-f(z)\mid.
\end{equation}
Using (\ref{3.2}) and (\ref{5.1}), we have
\begin{eqnarray}\notag
	\mid Q_n(z)-f(z)\mid\, &\leq&\,\underbrace{\sum\limits_{k=0}^{2n+1}\mid f(z_k)-F_n(z_k)\mid \mid A_k(z) \mid}_{N_1}+\underbrace{\mid F_n(z)-f(z)\mid }_{N_2}.\\\notag
\end{eqnarray}
We have
\begin{equation}\label{5.3}
	\mid Q_n(z)-f(z)\mid\,\leq \,N_1+N_2.
\end{equation}

From (\ref{2.14}) and (\ref{4.9}), we have
\begin{eqnarray}\label{5.4}
	N_1=\textit{\textbf{O}}\big(\omega(f,n^{-1})\log n\big).
\end{eqnarray}

From (\ref{2.14}), we have

\begin{eqnarray}\label{5.5}
	N_2=\textit{\textbf{O}}\big(n\,\omega(f,n^{-1})\big).
\end{eqnarray}
Using (\ref{5.4}) and (\ref{5.5}) in (\ref{5.3}), we get
\begin{eqnarray}\notag		
	\mid Q_n(z)-f(z) \mid \,=\, \textit{ \textbf{O}}\big(\omega(f,n^{-1})\log n\big).
\end{eqnarray}

Hence, Theorem \ref{thm1.1} follows.\\

\textbf{Data Availibility:}  Data sharing not applicable to this article as no data-sets were generated or analysed during the current study.\\

\textbf{Author contributions:}\\
\textit{Conceptualisation:} S. Bahadur, Varun ; \textit{Writing-Original Draft:}  Varun\\

\textbf{Conflicts of Interest:}  The authors declare no conflict of interest.


%
%


\begin{thebibliography}{HD}
%

%
%

%

	\bibitem{1}
Bahadur, S.: (0,0,1) interpolation on the unit circle. International Journal of Mathematical Analysis. 5, 1429-1434 (2011)

\bibitem{2} 
Berriochoa, E., Cachafeiro, A., Garc\'{i}a-amor, J. M.: An extension of Fej\'{e}r's condition for Hermite interpolation. Complex Analysis and Operator Theory. 6, 651–664 (2012)


\bibitem{3}
Berriochoa, E., Cachafeiro, A., D\'{i}az, J., Mart\'{i}nez Brey, E.: Rate of convergence of Hermite- Fej\'{e}r interpolation on the unit circle. Journal of Applied Mathematics. Article ID 407128, 8 pages (2013)  \url{https://doi.org/10.1155/2013/407128}



\bibitem{4} 
Berriochoa, E., Cachafeiro, A., D\'{i}az, J.: Hermite Interpolation on the Unit Circle Considering up to the Second Derivative. ISRN Mathematical Analysis. (2014) \url{http://dx.doi.org/10.1155/2014/808519}

\bibitem{5} 
Chen, W., Sharma, A.: Lacunary interpolation on some non-uniformly distributed nodes on the unit circle. Annales Universitatis Scientiarum Budapestinensis. 16, 69-82 (2004)	


\bibitem{6}
Daruis, L., Gonz\'{a}lez-vera, P.: A Note on Hermite -Fej\'{e}r Interpolation for the Unit Circle. Applied Mathematic Letters. 14, 997-1003 (2001)

\bibitem{7} 
Fej\'{e}r, L.: Über Interpolation. Gött. Nachr. 6, 66–91 (1916)

\bibitem{8}
Jackson, D.: Ueber die Genauigkeit der Annäherung stetiger Funktionen durch ganze rationale Funktionen gegebenen Grades und trigonometrische Summen gegebener Ordnung. Göttingen. (1911)

\bibitem{9}
Mills, T. M.:  Some techniques in Approximation theory. Math. Scientist. 5, 105-120 (1980)

\bibitem{10}
Prasad, J.: On Hermite and Hermite-Fejér interpolation of higher order. Demonstratio Mathematica. 26, 413-425 (1993)

\bibitem{11} 
Sung, H.S., Ko, D. H., Sakai, R.: Lp Convergence of Higher order Hermite or Hermite-Fej{\'e}r Interpolation polynomials with exponential-type weights ( II ). Global Journal of Pure and Applied Mathematics. 13, 7401–7426 (2017)

\bibitem{12}
Szabados, J., V\'{e}rtesi, P.: Interpolation of Functions. World Scientific Publishers. (1990)

\bibitem{13}
Szeg\H{o}, G.: Orthogonal Polynomials. Amer. Math. Soc. Coll. 23, (1975)

\bibitem{14} 
Xiang, S., He, G.: The Fast Implementation of Higher Order Hermite--Fejér Interpolation. SIAM Journal on Scientific Computing. 37, A1727–A1751 (2015)	




\end{thebibliography}


\end{document}